\documentclass[10pt]{article}
\usepackage{color,amsfonts, amsmath, amssymb,latexsym,amsthm}
\usepackage[margin=1.1in]{geometry}

\def\trans{{\;^\tau\!}}

\def\ui{{\underline i}}
\def\uj{{\underline j}}
\def\ua{{\underline a}}

\def\dim{\operatorname{dim}}

\newcommand{\m}{\mathfrak m}

\def\rank{\operatorname{rank}}

\def\ann{\operatorname{Ann}}

\newtheorem{theorem}{Theorem}[section]
\newtheorem{lemma}[theorem]{Lemma}
\newtheorem{corollary}[theorem]{Corollary}
\newtheorem{proposition}[theorem]{Proposition}
\newtheorem{remark0}[theorem]{Remark}
\newtheorem{example0}[theorem]{Example}
\newtheorem{definition}[theorem]{Definition}

\newenvironment{example}{\begin{example0}\rm}{\end{example0}}
\newenvironment{remark}{\begin{remark0}\rm}{\end{remark0}}

\newcommand{\propref}[1]{Proposition~\ref{#1}}
\newcommand{\thmref}[1]{Theorem~\ref{#1}}

\newcommand{\corref}[1]{Corollary~\ref{#1}}
\newcommand{\exref}[1]{Example~\ref{#1}}

\title{  \bf   Analytic Isomorphisms  of compressed local algebras
}
\author{\large   J. Elias
\thanks{Partially supported by  MTM2010-20279-C02-01}
\and \large M. E. Rossi
}

\date{\today}

\begin{document}
\maketitle
\begin{abstract}
In this paper we consider Artin local $K$-algebras with maximal length  in the class of Artin algebras  with given embedding dimension and socle type.  They have been widely studied by several authors, among others by   Iarrobino, Fr\"{o}berg and Laksov.    If the  local $K$-algebra is Gorenstein of socle degree     $3, $   then the authors proved that it  is   canonically   graded,  i.e.  analytically isomorphic to its  associated graded ring, see   \cite{ER12}.  This   unexpected result has been  extended    to compressed level   $K$-algebras of socle degree $3$   in \cite{DeS12}.   In this paper we end the investigation proving that  extremal  Artin  Gorenstein  local $K$-algebras  of socle degree $s \le 4$ are canonically graded, but the result does not  extend to extremal  Artin  Gorenstein  local rings  of socle degree $5  $  or to compressed level local rings of socle degree $4$ and type $>1.$
As a consequence we present   results  on   Artin   compressed  local  $K$-algebras  having a specified    socle type.
\end{abstract}

\section{Compressed algebras and Inverse System}

Let    $A=R/I $ be an Artin  local ring where $R=K[[x_1,\dots, x_n]] $ is the formal power series ring and $K$ is  an algebraically closed field of characteristic zero. We denote by $\m = {\mathcal M}/I $ the maximal ideal of $A$ where ${\mathcal M}=(x_1,\dots, x_n) $ and let $Soc(A) = 0 : \m $  be the  socle of $A.$ Throughout this paper we denote by $s$ the {\it{socle degree}}  of $A, $ that is the maximum integer $j$ such that $\m^j \neq 0.$ The {\it{type}} of $A$ is $t := \dim_K Soc(A). $

Let $A$ be of initial degree $v, $ that is $I \subseteq {\mathcal M}^v \setminus {\mathcal M}^{v+1}$  and socle degree $s.$  The {\it{socle type }} $E=E(A)=(0, \dots, e_{v-1}, e_v, \dots, e_s, 0, 0, \dots) $ of $A$ is the sequence $E$ of natural numbers where
$$ e_i := \dim_K ((0: \m) \cap \m^i / (0:\m) \cap \m^{i+1}).$$
Clearly $e_s > 0  $ and $e_j=0 $ for $j>s, $ but other conditions  are necessary on $E$ for being {\it{permissible}} for an Artin $K$-algebra of given initial degree  $v $ and socle degree $ s $ (see  \cite{Iar84}). Essentially the conditions require room for generators of $I$ of valuation  $v$ and assure that there are not generators of valuation $v-1.$

A  {\it{level algebra}} $A$  of socle degree $s $  (say also {\it{$s$-level}}), type $t $ and embedding dimension $n$ is an Artin quotient of $R$ whose socle is concentrated in a single degree (i.e. $Soc(A)= \m^s$) and satisfying   $ \dim_K Soc(A)=t.  $ Hence   $e_j=0$ for $j \neq s$ and $e_s=t.$ The Artin algebra is Gorenstein if  $t=1.$

It make sense to consider maximum length Artin algebras of given socle type $E  $ and they can be characterized  in terms of their Hilbert function.
The Hilbert vector of $A$  is denoted by $HF(A)=\{h_0, h_1, \dots, h_s\} $ where $h_i=h_i(A) = \dim_K \m^i/\m^{i+1} $ is the Hilbert function of $A. $ By  its definition, the Hilbert function of $A$ coincides with the Hilbert function of the corresponding associated graded ring $gr_{\m}(A)= \oplus_{i\ge 0} \m^i/\m^{i+1}. $
We say that the Hilbert function $HF(A)$ is maximal in the class of Artin  level algebras of given embedding dimension and socle  type, if  for each integer $i$, $h_i(A) \ge h_i(A') $ for any other Artin  algebra $A'$ in the same class.
The existence of a maximal $HF(A)$ was shown for graded algebras by A. Iarrobino \cite{Iar84}. In the general case by Fr\"{o}berg and Laksov \cite{FL84}, by Emsalem \cite{Ems78},  by A. Iarrobino and the first author in \cite{EI87} in the local case.

\begin{definition} An Artin  algebra $A=R/I $ of socle type $E$ is compressed if and only if it has maximal length $e(A)= \dim_K A$ among Artin  quotients of $R$ having socle type $E$ and embedding dimension $n. $
\end{definition}

  The maximality of the Hilbert function characterizes compressed algebras  as follows.
If $A$ is an Artin algebra (graded or local) of socle type $E, $ it is known  that for $i\ge 0,$
$$h_i(A) \le \min \{\dim_K  R_i, e_i \dim_K R_0 + e_{i+1} \dim_K  R_1 + \dots + e_s \dim_K  R_{s-i} \}.  $$
  Accordingly with \cite{Iar84}, Definition 2.4. B, we can rephrase  the previous definition in  terms of the Hilbert function.

\begin{definition}  \label{compressed} A local (or graded) $K$-algebra $A$ of socle degree $s, $   socle type $E  $    and initial degree $v$ is   {\it{compressed}} if

$$
h_i(A) =
\left\{
\begin{array}{ll}
\sum_{u = i}^s   e_u ( \dim_K R_{u-i})   & \text{ if }   i \ge v \\ \\
\dim_K R_i  & \text{ otherwise}
\end{array}
\right.
$$
In particular  a $s$-level algebra $A$  of   type $t $ and embedding dimension $n$ is   compressed if
$$ h_i(A) = min \{ \binom{n+i-1}{i}, \  t \binom{n+s-i-1}{s-i}\}.$$

\noindent
 If $t=1 $ and the above equality holds then $A$ is called an  {\it{extremal Gorenstein algebra}} or also {\it{compressed Gorenstein algebra}}.
 \end{definition}

 \vskip 2mm

It is clear that compressed algebras impose several restrictive numerical conditions on the socle sequence $E $ (see  \cite{Iar84}, Definition 2.2). For instance if $v$ is the initial degree of $A$, then
\begin{equation} \label{ev-1}  e_{v-1} = max \{0, \dim_K R_{v-1} - \sum_{u \ge v}(e_u \dim_K R_{u-(v-1)})\}. \end{equation}
If $s \ge 2(v-1), $ then it is easy to see that $e_{v-1}=0$ because
$\dim_K R_{s-(v-1)} \ge \dim_K R_{v-1}$.
 This is the case if $A$ is Gorenstein.

Y.H. Cho and I. Iarrobino proved that the family of graded compressed level  quotients of  $P=K[x_1, \dots, x_n]$ is  parameterized by an open dense subset of the Grassmanian of codimension $c $ quotients of $P_s$ and has dimension $c (\dim P_s - c)$ (see \cite{IarCho}).
 Boij studied the minimal free resolutions for compressed graded level algebras \cite{Boi99}, Stanley and Hibi have studied connections of these algebras with combinatorics.

\vskip 2mm
Accordingly with  Emsalem \cite[Pag. 408]{Ems78} we give the following definition.

\begin{definition}
A local algebra $(A, \m)$ is canonically graded if there exists a $K$-algebra isomorphism between $A$ and its associated graded ring $gr_{\m}(A). $
\end{definition}

\noindent
Notice that, as a consequence,  $A$ is   canonically graded if and only if  $A$ is analytically  isomorphic to a  standard graded $K$-algebra.  Remark that   a  local $K$-algebra
can  be considered graded in a non-standard way, but not canonically graded.
For instance, consider $A=K[[t^2,t^3]];  $ in this case $A $ is not analytically isomorphic to $gr_\m(A)\cong K[x,y]/(y^2)$  because $A$ is reduced and  $gr_\m(A) $ is not reduced, but
we remark that  $A\cong K[x,y]/(y^2-x^3)$
is graded   setting  $deg(x)=2$ and $deg(y)=3$.

 \bigskip
Actually it  is very rare that a local ring is isomorphic to its associated graded ring. If this is the case, several can be  the applications, see for example \cite{CENR}, \cite{ER12}, \cite{EV}.

 \bigskip

 If $A$ is an Artin   local (or graded) ring, the  dual module  $V= Hom(A, K) $ can be identified with a  maximal length submodule  of the polynomial ring  $P=K[y_1,\dots,y_n]  $ which is  closed under partial  derivation.
We recall the main facts and establish notation concerning  the inverse system of Macaulay
 in the  study   of  Artin   local rings $(A, \m).$
The reader should refer to \cite{Ems78} and \cite{Iar94} for an extended treatment.   It is known that
$P$ has an $R$-module structure  under   the following action
$$
\begin{array}{ cccl}
\circ: & R  \times P  &\longrightarrow &  P   \\
                       &       (f , g) & \to  &  f  \circ g = f (\partial_{y_1}, \dots, \partial_{y_n})(g)
\end{array}
$$
where $  \partial_{y_i} $ denotes the partial derivative with respect to $y_i.$
If  we  denote by $ x^{\alpha}= x_1^{\alpha_1} \cdots x_n^{\alpha_n} $ and
$ y^{\beta}= y_1^{\beta_1} \cdots y_n^{\beta_n} $ then
$$
x^{\alpha} \circ y^{\beta} =
\left\{
\begin{array}{ll}
\frac{\beta !}{(\beta-\alpha)!}\; y^{ \beta-  \alpha}  & \text{ if }  \beta_i \ge \alpha_i \text{ for } i=1,\cdots, n\\ \\
0 & \text{ otherwise}
\end{array}
\right.
$$
where $\frac{\beta !}{(\beta-\alpha)!}=\prod_{i=1}^n  \frac{\beta_i !}{(\beta_i-\alpha_i)!}$.

If $S \subseteq P,  $   then its annihilator   in $R$ is $ Ann_R(S)= \{ f \in R \ : \ f \circ g= 0 \ \mbox{ for \ \ all \ } g \in S\}.$  The $R$-submodule generated by $S$ in $P$ is closed under partial derivation, then $Ann_R (S) $ is an ideal of $R. $
  Conversely, for any ideal  $I\subset R$ we define the following $ R$-submodule of $P: $
$$
 {I^{\perp}}:=\{g\in P\ |\  \langle f, g \rangle = 0 \ \ \forall f  \in I \ \}= \{ g \in P\ |\  I \circ g = 0    \}.
$$
\noindent
Starting from an Artin local algebra   $R/I$ with  socle-degree  $s$, then $  {I^{\perp}} $ is generated by polynomials  of degree $\le s. $

If $S$ is generated by a sequence $\underline G := G_1, \dots, G_t $ of polynomials of $P, $ then we will write $Ann_R (\underline G) $ and $$A_{\underline G}= R/ Ann_R (\underline G) $$
where $Ann_R (\underline G) = \cap_{j=1}^t Ann_R(G_j).$ If $A_{\underline G} $ is compressed, then we may assume that each $A_{G_j}$ is an extremal Gorenstein algebra.
Notice that even if each $A_{G_j}$ is canonically graded, then  $A_{\underline G} $ is not necessarily canonically graded because it could not exists an uniform analytic isomorphism  (see Example \ref{t=2}).

F. S. Macaulay in \cite{mac16}
  proved that there exists a one-to-one correspondence between  ideals $I\subseteq R  $ such that $R/I$ is an Artin local ring and $R$-submodules  $M$ of $P$ which are finitely generated, see also  Emsalem
  \cite[ Section B, Proposition 2]{Ems78} and Iarrobino
  \cite[Lemma 1.2]{Iar94}.
 In this correspondence,   Artin  $K$-algebras $A=R/I$ of embedding dimension $n,$ socle degree $s$ and type $t$ correspond to $R$-submodules of $P$ generated by $t$ polynomials $G_1, \dots, G_t$ of degree $d_i$ (depending on the socle type) such that the leading     forms $G_1[d_1], \dots, G_t[d_t] $ of degree $d_i$ (write $G_i=G_i[d_i] + {\text{ lower \ terms}} \dots $)  are linearly independent.
If $ A= A_{\underline G}  $ where $\underline G := G_1, \dots, G_t  $  is $s$-level local algebra of  type $t $ ($d_i=s $ for all $i=1,\dots,t $),      then the associated graded ring $gr_{\m}(A) $ is not necessarily level and in general   $\dim_K Soc(gr_{\m}(A) ) \ge t. $ Actually  the graded $K$-algebra
 $$Q = P/Ann_P(\underline G[s])$$
(where $\underline G[s] := G_1[s], \dots, G_t[s]  $)  is a graded $s$-level algebra of type $t$ and it is the unique quotient of $gr_{\m}(A)   $  with socle degree $s$ and type $t.$ In particular $gr_{\m}(A)  $ is level if and only if $gr_{\m}(A) \simeq Q $ (see \cite{DeS12}).

\vskip 3mm
The following result was proved   in \cite[Proposition 3.7 and Corollary 3.8]{Iar84}.

 \begin{proposition} \label{G} The  compressed  local algebra $A$ whose dual module is generated by $G_1, \dots, G_t$ of degrees  $d_1, \dots, d_t $ has a compressed associated graded ring  $gr_{\m}(A) $ whose dual module is generated by $G_1[d_1], \dots, G_t[d_t],  $  the leading  forms of  $G_1, \dots, G_t. $  Conversely if   $gr_{\m}(A) $ is compressed, then $A$ is compressed and  $E(A)=E(gr_{\m}(A)). $
 \end{proposition}

\noindent It  is possible to compute the Hilbert function of $A=R/I $ via the inverse system.
We define the following $K$-vector space:
\begin{equation} \label{H1}
 (I^{\perp})_i := {\frac{ I^{\perp} \cap P_{\le i} +  P_{< i}}{ P_{< i}}}.
\end{equation}
Then  it is known  that
\begin{equation} \label{H2}
h_i(R/I) = \dim_K    (I^{\perp})_i .
\end{equation}

\noindent Notice that in   the graded setting    $A= R/Ann(\underline G), $   then
\begin{equation} \label{H3}
 h_i(R/I) =\dim_K    (I^{\perp})_i = \dim_K \langle \partial_{s-i} \underline G \rangle
 \end{equation}
where $\partial_{s-i} $ denotes the partial derivatives of order $s-i.$ We  will   translate  the above equality  in terms of the rank of suitable matrices associated to $\underline G.$

Let $\Omega=\{ \omega_i \} $ be the canonical basis     of  $ R/\mathcal M^{s+1}  $ as a $K$-vector space consisting of the standard monomials  $x^{\alpha}$ ordered by the {deg-lex order with $x_1>\cdots >x_n$} and,   then the dual basis with respect to the
action  $\circ $ is the basis $\Omega^*=\{ \omega_i^* \} $ of $P_{\le j} $  where
$$
(x^{\alpha})^* = \frac 1 {\alpha !} y^{\alpha},
$$
in fact $\omega_i^* (\omega_j)= \langle \omega_j ,  \omega_i^* \rangle=\delta_{ij}$,
where $\delta_{ij}=0$ if $i\neq j$ and $\delta_{ii}=1$.

\bigskip
Given a  form $G$ of degree $s$ and an integer $q\le s, $  we denote by $\Delta^q(G)$
the $\binom{n-1+s-q}{n-1}\times\binom{n-1+q}{n-1}$ matrix whose    columns
are  the coordinates of $\partial_{\ui}(G)$, $|\ui|=q$,
with respect $(x^{L})^* = \frac 1 {L !} y^{L}$, $|L|=s-q$ ($\ui$ and $L$ are $n$-uples of integers). We will denote by $(L, \ui)$ the corresponding position in the matrix $\Delta^q(G). $ In the following  $ L + \ui $ denotes the sum in $\mathbb  N^n.$

\bigskip
\begin{proposition}
\label{HFIS}
Let $G\in P=K[y_1,\cdots,y_n]$ be a form of degree $s. $  Then
$$
h_{s-i}(A_G)=\rank(\Delta^{i}(G)) \le min \{ \binom{n-1+s-i}{n-1}, \binom{n-1+i}{n-1}\}
$$
for $i=0,\cdots,s$.
The equality holds if and only if $A_G$ is compressed.
\end{proposition}
\begin{proof}
It is straight consequence of (\ref{H1}) and (\ref{H2}), and the definition of the
 matrix $\Delta^{i}(G)$ of size $ \binom{n-1+i }{n-1}\times  \binom{n-1+s-i }{n-1}. $
\end{proof}

\bigskip
\begin{proposition}
\label{derivative}
Let $G=  \sum_{|\uj|=s} \beta_\uj \; \frac{1}{\uj !} y^\uj$ be a degree $s$ homogeneous  polynomial.
Given an integer $q$,
the $(L, \ui)$-component of $\Delta^q(G)$, $|L|=s-q$, $|\ui|=q$, is
$$
\Delta^q(G)_{(L,\ui)}=\beta_{L + \ui }
$$
\end{proposition}
\begin{proof}
Let us consider the derivative
$$
\partial_{\ui}(G)=  \sum_{|\uj|=s, \; \uj\ge \ui} \beta_\uj \; \frac{1}{(\uj-\ui) !} y^{\uj-\ui}
$$
$|\ui|=q$;
$\uj\ge \ui$ means that $j_w\ge i_w$ for $w=1,\cdots, n$.
If we write $L=\uj - \ui$, then
$$
\partial_{\ui}(G)=  \sum_{|L|=s-q} \beta_{L+\ui} \;
\frac{1}{L !} y^{L}
$$
From this we get the claim.
\end{proof}


\bigskip
It is easy to deduce the following corollary which gives via (\ref{H3}) an alternative proof of the fact that a graded Gorenstein algebra $A_G$ has symmetric Hilbert function.

\begin{corollary}
\label{transpose}
Let $G $ be a form of $P$ of degree $s$
Given an integer $i \le s$,
then   $$ \Delta^i(G)= ^t\Delta^{s-i}(G) $$
where $^t $ denotes the transpose matrix.

\end{corollary}

\bigskip

With the previous notation, let  $A_{\underline{G}}$ be a graded level algebra   We can define for all integer $i \le s$,
\begin{eqnarray} \label{delta}
\Delta^i (\underline{G}) = \left(
\begin{tabular}{c}
$\Delta^i(G_1)$ \\

\hline
$\vdots$ \\
\hline
$\Delta^i(G_t)$
\end{tabular} \right)
\end{eqnarray}
which is a $t  \binom{n-1+s-i}{n-1} \times \binom{n-1+i}{n-1} $ matrix.
We get  the following result.

\begin{proposition} \label{HFComp} Let $A=A_{\underline G} $ be a  compressed  $s$-level local algebra of type $t.$  Then for every $i=1,\dots, s $
$$h_i(A) = rank ( \Delta^i (\underline{G}[s] ) = min \{\binom{n-1+i}{n-1}, t  \binom{n-1+s-i}{n-1}\}.$$
\end{proposition}
\begin{proof} By Proposition \ref{G} we know that $gr_{\m}(A)  $ is level compressed of socle degree $s$ and type $t. $ Since   $gr_{\m}(A)  $ is level if and only if $gr_{\m}(A) \simeq Q= P/Ann (\underline{G}[s] ),   $ the result follows by Proposition \ref{HFIS}.
\end{proof}


\bigskip
\section{Automorphisms of Artin local algebras  and associated matrices}

Given a $K$-algebra $C, $ we will denote by $Aut (C) $ the group of the automorphisms of $C$ as a $K$-algebra and by $Aut_K(C) $ as a $K$-vector space.
The  automorphisms   of $ R$ as  a $K$-algebra are well known. They act as replacement of $x_i$ by
$z_i $, $i=1,\cdots, n$, such that ${\mathcal M }= (x_1,
\dots, x_n)= (z_1, \dots, z_n). $
Actually, since $ {\mathcal M }^{s+1}  \subseteq I, $ we are interested in the   automorphisms of $R/ {\mathcal M}^{s+1} $   of $K$-algebras  induced by
the projection $\pi:R\longrightarrow R/{\mathcal M }^{s+1}.$
Clearly    Aut$(R/ {\mathcal M}^{s+1})  \subseteq $   Aut$_K (R/ {\mathcal M}^{s+1}).   $
For  all $p \ge 1$,
$I_p$ denotes  the   identity matrix of order $\binom{n+p-1}{p}.$

 For any
$\varphi  \in  Aut_{K}(R/{\mathcal M }^{s+1}) $ we may associate   a    matrix $ M(\varphi)$   with respect to the basis $\Omega$ of size $ r= dim_K (R/{\mathcal M }^{s+1}) =   {{n+s} \choose s}.  $
\noindent Given  $I $ and $ J $   ideals of $R$ such that $\mathcal M^{s+1}\subset I, J, $ there
  exists an  isomorphism of  $K$-algebras
$$
\varphi : R/I \to R/J
$$
if and only if $ \varphi  $ is canonically induced  by  a $K$-algebra automorphism   of $ R/{\mathcal M }^{s+1} $  sending  $I/{\mathcal M }^{s+1}$ to $J/{\mathcal M }^{s+1}.$
In particular $\varphi$ is an isomorphism of $K$-vector spaces. Passing to $Hom(\ ,K), $   $\varphi $  corresponds to
$$
\varphi^* :  {J^{\perp}}   \to  {I^{\perp}}
$$
where   $^t M (\varphi) $ is the matrix associated to $\varphi^* $
with respect to the basis $\Omega^*  $ of $P_{\le s}.$

\noindent
We denote by ${\mathcal R}$ the  subgroup of  $Aut_{K}(P_{\le s})$    (automorphisms of $ P_{\le s} $ as a $K$-vector space)
   represented     by
 the matrices $^t M (\varphi) $ of  $Gl_r(K)$ with $\varphi\in Aut (R/{\mathcal M }^{s+1}) $.

By Emsalem,   \cite[Proposition 15]{Ems78},
the classification,  up to analytic isomorphism,   of the Artin  local $K$-algebras of multiplicity $d, $ socle degree $s$ and embedding dimension $n$ is equivalent to the classification,
 up to the action of ${\mathcal R},  $ of the $K$-vector subspaces of $P_{\le s}$  of dimension $d, $ stable by derivations and containing $P_{\le 1 }=K[y_1, \dots, y_n]_{\le 1}.$

 \bigskip

 Our goal is to translate the study of analytic isomorphisms   of   Artin level $K$-algebras  $A=R/I $  of socle $s$ and type $t$  in terms of the corresponding dual polynomials  of degree $s $ in $P $ in an  effective computational framework.  This section is quite technical and the idea is   the generalization of the method used  in \cite{ER12}.   The main goal is   Corollary \ref{rankM} and,  in view of it,    the machinery  could  be presented  assuming
$s\le 4.$ Actually the presentation does not improve substantially and one   loses   the control of  the general case that could be useful for further investigations.

\bigskip

\noindent  Let   ${\underline F} = F_1, \dots, F_t,  $ respectively
${\underline G} = G_1, \dots, G_t,  $  be polynomials of degree $s. $ Let
 $\varphi  \in Aut (R/ \mathcal M ^{s+1}), $  from the previous facts we have
 \begin{equation}
\varphi(A_{\underline F})=A_{\underline G} \ \ \text{ if\  and \ only\  if } \  (\varphi^*)^{-1}(\langle {\underline F}  \rangle_R)=\langle {\underline G} \rangle_R.
 \end{equation}

  If    $F_i  =  b_{i1} w_1^*+ \dots b_{ir} w_r^*  \in P_{\le s}, $ then we will denote the {\it row vector} of the coefficients of the polynomial with respect to the basis  $\Omega^*$ by
    $$  [F_i]_{\Omega^*} = (b_{i1}, \dots, b_{ir}).$$


\noindent If  there exists   $\varphi  \in Aut (R/ \mathcal M ^{s+1}) $  such that
\begin{equation} \label{matrix}  [  G_i]_{\Omega^*} M(\varphi ) =   [  F_i]_{\Omega^*}, {\text{\ for \ every }} i=1, \dots, t, {\text{ then \ }} \varphi(A_{\underline F})=A_{\underline G}
\end{equation}

\noindent

Let $s$ be a positive  integer, the aim of this section is to provide a structure of the matrix $M(\varphi)$ associated
to special  $K$-algebra isomorphisms $\varphi \in Aut (R/{\mathcal M }^{s+1}). $
 Let $\varphi_{s-p} $ be   an automorphism of $R/\mathcal M^{s+1}$ such that
$\varphi_{s-p} = Id$ modulo $\m^{p+1}$,  with $1 \le p \le s,  $  that is
\begin{equation} \label{varphi}
\varphi_{s-p}(x_j)=x_j+ \sum_{|\ui|=p+1} a_{\ui}^j x^{\ui}+ \text{higher terms}
\end{equation}
for $j=1, \dots, n $ and  $ a_{\ui}^j \in K$ for each $n$-uple $ \ui$ such that $|\ui|=p+1.$
In the following we will denote
$\ua: = (a_{\ui}^1, |\ui|=p+1  ; \cdots; a_{\ui}^n, |\ui|=p+1 )\in K^{n \binom{n+p}{n-1}}. $

The  matrix associated to $\varphi_{s-p}, $ say $ M(\varphi_{s-p}),  $  is an  element of $Gl_r(K)$, $r=\binom{n+s}{s+1}$, with  respect to the basis $\Omega$
of $R/\mathcal M^{s+1}$.
We write
$ M(\varphi_{s-p})=(B_{i,j})_{0\le i, j \le  s}$
where $B_{i,j}$ is a $\binom{n+i-1}{i}\times \binom{n+j-1}{j}$ matrix of the coefficients of monomials of degree $i$ appearing in $\varphi(x^{\uj}) $ where $\uj=(j_1, \dots, j_n) $ such that $|\uj|=j.$

It is easy to verify that:
$$
B_{i,j}=
\left\{
  \begin{array}{ll}
    0, & 0\le i < j \le s, \text{ or } j=1, i=1,\cdots, s,\\ \\
    I_i,& i=j=0,\cdots, s,\\ \\
    0,  & j=s-p,\cdots, s-1, i=j+1,\cdots, s, \text{ and } (i,j)\neq (s,s-p).
  \end{array}
\right.
$$

\medskip

The matrix $M(\varphi_{s-p})$ has the following structure
$$ M(\varphi_{s-p})=\left(
\begin{array}{c|c|c|c|c|c|c|c|c}
1 & 0  & \cdots &0 &   0 & 0& 0 &0\\      \hline
0 & I_1&0  &0   &0  &0& 0 & \vdots\\   \hline
0 &0& I_2 &0&  0  &0& 0& \vdots\\   \hline
\vdots &\vdots  &0 &\ddots   &\vdots  & \vdots& \vdots& \vdots\\   \hline
 0 & B_{p+1,1}&0  &\dots &   I_{s-p }&0& 0& \vdots\\   \hline
0 &  \dots &B_{p+2,2}&0     & 0 &I_{s-p+1}&0& \vdots \\   \hline
0 & \dots &\cdots&\ddots &   \vdots &0 &\ddots &0\\      \hline
  0 & B_{s,1}&B_{s,2} & \dots    & B_{s,s-p} &0& \dots& I_s
\end{array}
\right)
$$
The entries of $B_{p+1,1}, B_{p+2,2},  \dots, B_{s,s-p} $ are linear forms in the variables
$a_{\ui}^j$, with  $|\ui|=p+1$, $j=1,\cdots,n$.
We are mainly interested in $ B_{s,s-p}   $ which is a   $\binom{n+s-1}{s}\times \binom{n+s-p-1}{s-p}$ matrix whose columns correspond to $x^W$ with $|W|=s-p $ and the rows correspond to the coefficients of $x^L$ with $|L|=s $ in $\varphi(x^W).$
 One has
$$
\varphi(x^W)=x^W + \sum_{j=1}^n w_j (\sum_{|\ui|=p+1} a_{\ui}^j x^{W-\delta_j+\ui}) + \dots
$$
where $W=(w_1,\dots,w_n)\in\mathbb N^n$ with $|W|=s-p$, here
$\delta_j$  is the $n$-uple with $0$-entries but $1$ in position $j$. We remark that $|W-\delta_j+\ui|=s. $
Then
the entry of $B_{s,s-p}$ corresponding to the $L$ row, $|L|=s$,  and $W$ column, $|W|=s-p$,
is
\begin{equation}
\label{entriesB}
(B_{s,s-t})_{L,W}=\sum_{W-\delta_j+\ui=L} w_j a_{\ui}^j.
\end{equation}

\bigskip
Let  $F, G$ be  polynomials of degree $s$ of $P   $   and let $\varphi_{s-p} $ be a  $K$-algebra isomorphism   of type  (\ref{varphi})  sending   $A_F$ to $A_G. $ We denote by $F[j] $ (respectively $G[j]$) the homogeneous component  of degree $j$ of $F$ (respectively of $G$), that is   $F=F[s]+F[s-1]+ \dots  $ ($G=G[s]+G[s-1]+ \dots$).

\par \noindent  By  (\ref{matrix}) we have  \begin{equation} \label{hom} [G]_{\Omega^*} M(\varphi_{s-p}) = [F]_{\Omega^*}, \end{equation}  in particular we deduce
\begin{equation}
\label{killing}
[F[j]]_{\Omega^*}=
\left\{
  \begin{array}{ll}
    [G[s-p]]_{\Omega^*}+  [G[s]]_{\Omega^*}B_{s,s-p},& j=s-p, \\ \\
    {[G[j]]}_{\Omega^*},  & j=s-p+1,\cdots, s.
  \end{array}
\right.
\end{equation}

\medskip

We are going to study  $ [G[s]]_{\Omega^*}B_{s,s-p}.  $
 Let $[\alpha_\ui ]$ be the vector of  the coordinates of $G[s]$ w.r.t.  $\Omega^*$, i.e.
$$
G[s]=  \sum_{|\ui|=s} \alpha_\ui \; \frac{1}{\ui !} y^\ui;
$$
the entries of
$[G[s]]_{\Omega^*}B_{s,s-p}$
are bi-homogeneous forms  in  the components of $[\alpha_\ui ]$ and   $\ua=(a_{\ui}^1, \dots, a_{\ui}^n) $ such that $|\ui|=p+1$
of bi-degree $(1,1)$.
Hence there exists a matrix $ M^{[s-p]}(G[s])$ of size
$  \binom{n-1+s-p}{n-1} \times  n \binom{n+p}{n-1} $
and  entries   in  the $K[\alpha_\ui ] $ such that

\begin{equation}
\label{matrixM}
\trans( [\alpha_\ui ]  B_{s,s-p}) = M^{[s-p]}(G[s])  \trans \ua
\end{equation}
where $\trans\ua $ denotes the transpose of the row-vector $\ua.$ We are going to describe the entries of $ M^{[s-p]}(G[s]).$
We label the columns of $ M^{[s-p]}(G[s])$ with the set of indexes
$(j,\ui)$, $j=1,\cdots,n$, $|\ui|=p+1$, corresponding to the
entries of $\ua = (a_{\ui}^1, |\ui|=p+1  ; \cdots; a_{\ui}^n, |\ui|=p+1 )\in K^{n \binom{n+p}{n-1}}$.

\bigskip
\begin{lemma}
\label{entriesMG}
The entry of $ M^{[s-p]}(G[s])$ corresponding to the $W$-row, $|W|=s-p$, and column
$(j,\ui)\in \{1,\cdots, n\}\times \{\ui ; |\ui|=p+1\}$ is
$$
M^{[s-p]}(G[s])_{W, (j,\ui)}= w_j\  \alpha_{W-\delta_j + \ui}.
$$
\end{lemma}
\begin{proof}
Given $W$, $|W|=s-p$, the coordinate of $[G[s]]_{\Omega^*}B_{s,s-p}$ with respect $(x^W)^*$ is, by (\ref{entriesB}),
\begin{multline*}
([G[s]]_{\Omega^*}B_{s,s-p})_{(x^W)^*}=
\sum_{|L|=s} (B_{s,s-p})_{W,L}  \ \alpha_L=\\
=\sum_{|L|=s} \left(\sum_{W-\delta_j+\ui=L} w_j \ a_{\ui}^j\right) \alpha_L=
\sum_{j, |\ui|=p+1} w_j \ \alpha_{W-\delta_j+\ui}\  a_{\ui}^j
\end{multline*}
so the entry of $ M^{[s-p]}(G[s])$ corresponding to the $W$ row, $|W|=s-p$, and column
$(j,\ui)\in \{1,\cdots, n\}\times \{\ui ; |\ui|=p+1\}$ is
$\ w_j \alpha_{W-\delta_j + \ui}$.
\end{proof}

\bigskip
The goal is now to present  a structure of  $M^{[s-p]}(G[s])$  in terms of $G[s].$ In particular we will prove that the rank of  $M^{[s-p]}(G[s])$ can be expressed in terms of the Hilbert function of $A_{G[s]}. $ We need  further notations.

\noindent   For every $i=1,\cdots,n, $ we denote   $S^i_p$ the set of monomials $x^{\alpha}$ of degree $p$ such that $x^{\alpha}\in x_i (x_i,\cdots,x_n)^{p-1}, $  hence  $\#(S^i_p)={p-1+n-i \choose p-1}. $
 By definition  $S^1_p\cup \cdots \cup S^n_p$ is the   set of monomials of degree $p$    and
  \begin{equation}
   \label{SS}
 \textrm{
  the last\; }
  \binom{p-2+n-i}{p-1}
  \textrm{\;   elements of  \; }
  S^i_p \textrm{\;   correspond \ to  \; }
   \frac{x_i}{x_{i+1}} S^{i+1}_p,
  \end{equation}

  \noindent  For instance $S^1_3=\{x_1^3, x_1^2 x_2, x_1^2 x_3, x_1 x_2^2, x_1x_2x_3, x_1x_3^2\}$,
  $S^2_3=\{x_2^3,x_2^2x_3, x_2x_3^2\}$, $S^3_3=\{x_3^3\}$.
  We write $\log(x^{\alpha})=\alpha$ for all $\alpha \in \mathbb N^n$.

\begin{lemma}
\label{structureMG}
The matrix $M^{[s-p]}(G[s])$ has the following upper-diagonal structure
$$ M^{[s-p]}(G[s])=\left(
    \begin{array}{l|l|l|l|l}
M_1&  * &  \cdots & * & * \\      \hline
0 & M_2&   \cdots & * & * \\   \hline
\vdots & \vdots&  \vdots & \vdots & \vdots  \\      \hline
0 & 0& 0 &  M_{n-1}&*  \\ \hline
0 & 0& 0 & 0&  M_n
\end{array}
  \right)
$$
where $ M_j$ is a matrix of size $\binom{s-p-1+n-j}{s-p-1}\times \binom{n+p}{n-1}, $  $j =1,\cdots,n$, defined as follows:
the entries of $M_j$ are the entries of $M^{[s-p]}(G[s])$ corresponding to the
rows $W\in \log(S^j_{s-p})$ and columns $(j,\ui)$, $|\ui|=p+1$.
We label the entries of $M_j$ with respect to these multi-indexes.
Then it holds:
\begin{enumerate}
\item[(i)]
for all $W=(w_1,\cdots,w_n)\in \log(S^1_{s-p})$ and $\ui$, $|\ui|=p+1$,
$$
w_1 \Delta^{p+1}(G[s])_{(W-\delta_1,\ui)} ={M_1}_{(W,(1,\ui))},
$$
\item[(ii)]
for all $j=1,\cdots, n-1$, $W\in \log(S_{s-p}^{j+1})$,
$$
M_{j+1, (W,(j+1,*))}= w_{j+1} M_{j,(L,(j,*))}
$$
with
$L=\delta_j + W - \delta_{j+1}$,
\end{enumerate}
\end{lemma}
\begin{proof}
First we prove that $M^{[s-p]}(G[s]$  has the upper-diagonal structure as in the claim.
Since  the entry of $ M^{[s-p]}(G[s])$ corresponding to the $W$ row, $|W|=s-p$, and column
$(j,\ui)\in \{1,\cdots, n\}\times \{\ui ; |\ui|=p+1\}$ is
$w_j  \alpha_{W-\delta_j + \ui}$, this entry is zero  if
$W\in \log(S_{s-p}^t)$ and $j<t \le n$.

\noindent
$(i)$
Notice that the set of multi-indexes $L=W-\delta_1$, $|W|=s-p$, agrees
with the set of  multi-indexes of degree $s-p-1$, see (\ref{SS}).
Hence by \propref{derivative} and \propref{entriesMG} we have
$$
w_1 \Delta^{p+1}(G[s])_{(W-\delta_1,\ui)} = w_1 \alpha_{W-\delta_1+\ui}=
{M_1}_{(W,(1,\ui))}.
$$

\noindent
$(ii)$
If $|\ui|=p+1$ and $L=\delta_j + W - \delta_{j+1}$, then, \propref{entriesMG},
$$
M_{j+1, (W,(j+1,*))}= w_{j+1} \alpha_{W-\delta_{j+1} + \ui}=
w_{j+1} \alpha_{L-\delta_j + \ui}=
w_{j+1} M_{j,(L,(j,*))}.
$$
\end{proof}

As a   consequence of $(i)$ and $(ii), $  the  matrix   $ \Delta^{p+1}(G[s])$ is strongly  involved in the computation of the rank of $M^{[s-p]}(G[s].$

\begin{corollary} \label{rankM} If $s\le 4$ then $\text{ rank }(M^{[s-p]}(G[s]))$ is maximal if and only if
$\text{ rank }(\Delta^{p+1}(G[s]))$ is maximal.
\end{corollary}
\begin{proof}
Notice that $M^{[s-p]}(G[s])$ has an upper-diagonal structure
where the rows of the diagonal blocks $M_j$ are a subset of the  rows of the first block matrix $M_1$.
Let us assume  that the number of rows of $M_1$ is not bigger than  the number of columns of $M_1, $ as a consequence the same holds for $M_j$ with $j>1.$
Then we can compute the rank of $M^{[s-p]}(G[s]) $ by rows, so $\text{ rank }(M^{[s-p]}(G[s]))$ is maximal if and only if
$\text{ rank }(\Delta^{p+1}(G[s]))$ is maximal.
Since  $M_1$ is a
$\binom{s-p-2+n}{s-p-1}\times \binom{n+p}{n-1}$ matrix,
if $\binom{n+s-p-2}{s-p-1}=\binom{n+s-p-2}{n-1}\le\binom{n+p}{n-1}$ we get the result.
This inequality is equivalent to $n+s-p-2 \le n+p$, i.e. $s\le 2 p +2$, since $p\ge 1$ we get that
$s\le 4$.
\end{proof}

\medskip
 \exref{s5} shows that  Corollary \ref{rankM} fails for  $s=5$.

\bigskip
We may generalize the previous facts to a  sequence   $  \underline G=G_1, \dots, G_t $ of    polynomials of degree $s$ of $P.    $
 Let $\varphi_{s-p} $ be a  $K$-algebra isomorphism   of type  (\ref{varphi})  sending   $A_{\underline F}$ to $A_{\underline G} $ where  ${\underline F}=F_1, \dots,F_t . $ In particular we assume that,   as  in (\ref{hom}),
$$[G_r]_{\Omega^*} M(\varphi_{s-p}) = [F_r]_{\Omega^*}, $$
for every $r=1, \dots, t.$
We deduce the analogous of (\ref{killing}) and we restrict our interest to
$$   [\underline G[s]]_{\Omega^*}B^{\oplus t}_{s,s-p}$$
where
\begin{equation*}
B^{\oplus t}_{s,s-p}
:= \left(
\begin{tabular}{c}
$B_{s,s-p}$ \\

\hline
$\vdots$ \\
\hline
$B_{s,s-p}$
\end{tabular} \right)
\end{equation*}
obtained by gluing  $t$ times the matrix $B_{s,s-p}$   and  where  $ [\underline G[s]]_{\Omega^*}$ is the row $([G_r[s]]_{\Omega^*} : r=1, \dots, t).$
 Accordingly with (\ref{matrixM}),  it is  defined the matrix    $ M^{[s-p]}(G_r[s])$ of size
$  \binom{n-1+s-p}{n-1} \times  n \binom{n+p}{n-1} $
and  entries  depending on $  [\underline G[s]]_{\Omega^*}$  such that

\begin{equation*}
^t([G_r[s]]_{\Omega^*} B_{s,s-p}) = M^{[s-p]}(G_r[s])  \ ^t \ua
\end{equation*}

\noindent If we define
\begin{equation}
\label{multi}
M ^{[s-p]}({\underline G}[s])
:= \left(
\begin{tabular}{c}
$M^{[s-p]}(G_1[s])$ \\

\hline
$\vdots$ \\
\hline
$M^{[s-p]}(G_t[s])$
\end{tabular} \right)
\end{equation}
which is a $t  \binom{n-1+s-p}{n-1} \times n \binom{n+p}{n-1} $  matrix,     we get
\begin{equation}
\label{matrixMM}
\trans([\underline G[s]]_{\Omega^*}   B^{\oplus t}_{s,s-p}) = M^{[s-p]}({\underline G}[s])  \trans \ua.
\end{equation}
The matrix  $M ^{[s-p]}({\underline G}[s]) $ has the same shape of $M ^{[s-p]}({G}[s]),  $ already described in Lemma \ref{structureMG} and its  blocks correspond to suitable submatrices of $(\Delta^{p+1}(\underline G[s]))$ (see   (\ref{delta})).
Hence we have an analogous to (\ref{killing}) for the level case
\begin{equation}
\label{multikilling}
[F_r[j]]_{\Omega^*}=
\left\{
  \begin{array}{ll}
    [G_r[s-p]]_{\Omega^*}+ \ua \; \;  ^\tau(M^{[s-p]}({\underline G}[s])),& j=s-p, \\ \\
    {[G_r[j]]}_{\Omega^*},  & j=s-p+1,\cdots, s.
  \end{array}
\right.
\end{equation}
for all $r=1, \dots, t.$


\bigskip
\section{Compressed  Gorenstein  algebras}

Let $A=A_G$ be an compressed Artin Gorenstein local $K$-algebra $A$ of socle degree $s$. We recall that,
by Proposition \ref{HFComp}, if  $G=G[s]+G[s-1]+ \dots,  $ then
$$h_i(A) = rank ( \Delta^i ({G[s]}) = \min \{\binom{n-1+i}{n-1},   \binom{n-1+s-i}{n-1}\} $$
for every $i=1, \dots, s.$

\bigskip
\begin{theorem}
\label{cangradedm=n}
Let $A $ be an extremal  Artin  Gorenstein local $K$-algebra.
If $s\le 4$ then $A $ is canonically graded.
\end{theorem}
\begin{proof}
 Let  $A$ be a extremal   Artin  Gorenstein local $K$-algebra of socle degree $s\ge 2 $ and embedding dimension $n.$
Then $A=A_G$ with $G \in P=K[y_1, \dots, y_n] $ a polynomial of degree $s $ and $gr_{\m}(A)= P/\ann(G[s]) $
is an   extremal   Gorenstein graded algebra of socle degree $s\ge 2 $ and embedding dimension $n $ (see Proposition \ref{G}).

The main result of \cite{ER12} shows that if $s\le 3$ then $A$ is canonically graded.
Let assume $s=4,  $  then the Hilbert function is $\{ 1, n, {n+1 \choose 2}, n, 1\}.  $ Because $A_{G[4]} $ is an  extremal   Gorenstein algebra with the same Hilbert function of $A, $ we may assume $G=G[4]+G[3]. $ In fact  $P_1, P_2 \subseteq <G[4]>_R $ because of  (\ref{H1}) and, as a consequence,  it easy to see that   $ <G[4]+G[3]>_R = <G[4]+G[3]+ G[2]+...>_R.  $

So we have to prove that, however we fix $G[3]$, there exists an automorphism   $\varphi  \in Aut(R/\mathcal M^{5})  $  such that
$$A_{G} \simeq A_{G[4]}.$$
We consider for every $j=1, \dots, n $
$$
\varphi(x_j)=x_j+ \sum_{|\ui|=2} a_{\ui}^j x^{\ui}+ \text{higher terms}
$$
If $A_F=\varphi^{-1}_{3}(A_G), $ then from (\ref{killing}) and (\ref{matrixM}) we get

\begin{equation} \label{system}
\begin{array}{l}
[F[3]]_{\Omega^*}= [G[3]]_{\Omega^*}+ \ua \trans(M^{[3]}(G[4])) \\ \\
{[F[4]]}_{\Omega^*}= [G[4]]_{\Omega^*}
\end{array}
\end{equation}
where  $\ua=(a_{\ui}^1, \dots, a_{\ui}^n). $
By  Proposition \ref{HFComp} and  \corref{rankM}, we know that   the matrix  $M^{[3]}(G[4])$ has maximal rank and it coincides with the number of the rows, so
there exists a solution  $\ua \in K^n $  of (\ref{system})  such that $F[3]=0$ and $F[4]=G[4]$.
 \end{proof}

\bigskip
Let $A$ be a local $K$-algebra of embedding dimension $n $ and socle degree $s.$ By looking at the dual module, it is clear that if $s\le 2, $ then $A$ is graded because the dual module can be generated by homogeneous polynomials of degree at least two. The aim is now  to list the local compressed algebras of embedding dimension $n, $ socle degree $s$ and socle type $E =(0, \dots, e_{v-1}, e_v, \dots, e_s, 0, 0, \dots) $  which are canonically graded. Examples will prove  that the following result cannot be extended to higher socle degrees.

\bigskip
\begin{theorem}
\label{cangradedlevel}
Let $A $ be a compressed  Artin  $K$-algebra of  embedding dimension $n, $ socle degree $s$ and socle type $E.$ Then $A$ is canonically graded  in the following cases:
\begin{enumerate}
\item[(1)]
 $s \le 3$,

\item[(2)]
$ s=4$ and $e_4=1$,

\item[(3)]
 $s = 4$ and $n=2$.
\end{enumerate}
\end{theorem}
\begin{proof}
Since a local ring with  Hilbert function   $\{1,n,t\}$ is always graded (the dual module can be generated by quadratic forms), we may assume $s \ge 3.$

If $s=3$ and $A$ is level compressed, then  $A$ is canonically graded by  \cite{DeS12}.  If $A$ is not necessarily level, but compressed, then by (\ref{ev-1})  the socle type is $\{0,0, e_2, e_3\} $ and the Hilbert function is $\{1, n, h_2, e_3\}  $ where $h_2 = \min \{ \dim_K R_2, e_2+ e_3 n\}.$    Because $gr_{\m}(A) = P/I^*$ has embedding dimension $n, $   then  $P_{\le 1} \subseteq (I^*)^\bot. $
Then we may assume that in any system of coordinates $I^\bot$ is generated by  $e_2$ quadratic forms    and $e_3$ polynomials $G_1, \dots, G_{e_3} $ of degree $3.$
Then the result  follows because $R/\ann_R(G_1, \dots, G_{e_3}) $ is a  $3$-level compressed algebra  of type $e_3$ and  hence canonically graded.

Assume $s=4 $ and $e_4=1.$ We recall that if $A$ is Gorenstein, then  the result follows by Theorem \ref{cangradedm=n}. Since $A$ is compressed, then by (\ref{ev-1})  the socle type is $(0,0, 0, e_3,1). $ This means that $I^\bot$ is generated by $e_3$ polynomial of degree $3$ and one polynomial of degree $4.$  Similarly  to the above part,  because $P_{\le 2}  \subseteq  (I^*)^\bot, $   $I^\bot$ can be  generated by $e_3$ forms  of degree $3$  and one polynomial of degree $4.$ As before the problem is reduced to the Gorenstein case with $s=4 $ and the result follows.

Assume $s=4 $ and $n=2$. If $e_4=1, $ then we are in case $(2).$ If $e_4>1, $ because $A$ is compressed, the possible socle types are:  $E_i=(0,0, 0, 0, i)$ with $i=2,\cdots, 5 $ and since $A$  is compressed, the  corresponding Hilbert function is $\{1, 2, 3, 4, i\}.$
In each case  $A$ is graded   because the Hilbert function forces the dual module to be  generated by   forms of degree four.

\end{proof}

\bigskip
The following example shows that \thmref{cangradedm=n} fails if $A$ is Gorenstein of socle degree $s=4, $ but not compressed, i.e.   the Hilbert function is not maximal.

 \begin{example}
 Let $A$ be an Artin   Gorenstein local $K$-algebra   with Hilbert function $HF_A=\{1,2,2,2,1\}$.  The local ring is called almost stretched and a classification can be found in \cite{EV}. In this case   $A $ is isomorphic to one and only one of the following  rings :
\begin{enumerate}
\item[(a)]
$A=R/I$ with $I=(x_1^4, x_2^2) \subseteq  R=K[[x_1, x_2]], $ and $I^{\perp}=\langle y_1^3 y_2\rangle. $ In this case  $A$ is canonically graded,
\item[(b)]
$A=R/I$ with $I=(x_1^4, -x_1^3 + x_2^2)\subseteq  R=K[[x_1, x_2]], $  and $I^{\perp}=\langle y_1^3 y_2+ y_2^3\rangle$.
The associated graded ring is of type $(a)$ and it is  not  isomorphic to $R/I. $ Hence $A$ is not canonically graded.
\item[(c)]
$A=R/I$ with $I=(x_1^2+x_2^2, x_2^4) \subseteq  R=K[[x_1, x_2]], $   and $I^{\perp}=\langle y_1 y_2(y_1^2-y_2^2) \rangle. $ In this case  $A$ is   graded.
\end{enumerate}
 \end{example}

\medskip
The following example  shows   that \thmref{cangradedm=n} cannot be extended to extremal  Gorenstein algebras of socle degree $s=5$.

\begin{example}
\label{s5}
Let us consider the  ideal
$$
I=(x_1^4, x_2^3 - 2 x_1^3 x_2)\subset R=K[[x_1,x_2]].
$$
The quotient $A=R/I$ is an  extremal  Gorenstein  algebra
with  $HF_{A}=\{1,2,3,3,2,1\}$, $I^*=(x_1^4,x_2^3)$ and $I^{\perp}=\langle y_1^3 y_2^2 +   y_2^4\rangle$.
We will prove that $I$ is not isomorphic to $I^*$.
Assume that there exists an analytic  isomorphism $\varphi$ of $R$ mapping $I$
into $I^*$.
It is easy to see that  Jacobian matrix of $\varphi$ is diagonal because $(I^*)^{\perp} = <y_1^3 y_2^2>. $
We perform the computations modulo $(x_1,x_2)^5$, so we only have to consider the following coefficients of $\varphi$
$$
\left\{
  \begin{array}{ll}
\varphi(x_1)= a x_1 + \dots \\ \\
\varphi(x_2)= b x_2 + i x_1^2 + j x_1 x_2 + k x_2^2+ \dots
  \end{array}
\right.
$$
where $a, b $ are units, $i, j, k \in K$.
After  the isomorphism $x_1 \rightarrow  1/a x_1$, $x_2 \rightarrow  1/ bx_2, $  we may assume
$a=b=1$.
Then we have
$$
I^*= \varphi(I )= (x_1^4,  x_2^3 -2  x_1^3 x_2 + 3  i x_1^2 x_2^2 +  3  j x_1 x_2^3 +
 3  k x_2^4) \text{\quad  modulo } (x_1,x_2)^5.
$$
Hence there exist  $\alpha \in K ,\beta \in R$ such that
$$
 x_2^3 -2  x_1^3 x_2 + 3  i x_1^2 x_2^2 +  3  j x_1 x_2^3 +
 3  k x_2^4 = \alpha x_1^4 + \beta  x_2^3  \text{\quad   modulo } (x_1,x_2)^5.
$$
From this equality we deduce $\alpha=0$ and
$$
2  x_1^3 x_2 = x_2^2( x_2  + 3 i x_1^2 + 3  j x_1 x_2 +  3  k x_2^2 - \beta x_2  )  \text{ \quad modulo } (x_1,x_2)^5,
$$
a contradiction, so $I$ is not isomorphic to $I^*$.

It is interesting to notice that we can get the same conclusion  by following the line of the proof of Theorem \ref{cangradedm=n}.   Let $\varphi $ as above sending $I$ into $I^*. $  If we denote by $(z_i)_{i=1,\dots,6}$ the coordinates of an homogeneous form $G[5] $  of degree $5$ in $y_1, y_2$ with respect $\Omega^*, $ then by Lemma \ref{entriesMG},
the matrix $M^{[4]}(G[5])$ ($s=5, p=1$)  has  the following shape
$$
\left(
\begin{array}{cccccc}
 4 z_1 & 4 z_2 & 4 z_3 & 0 & 0 & 0 \\
 3 z_2 & 3 z_3 & 3 z_4 & z_1 & z_2 & z_3 \\
 2 z_3 & 2 z_4 & 2 z_5 & 2 z_2 & 2 z_3 & 2 z_4 \\
 z_4 & z_5 & z_6 & 3 z_3 & 3 z_4 & 3 z_5 \\
 0 & 0 & 0 & 4 z_4 & 4 z_5 & 4 z_6
\end{array}
\right)
$$
 In our  case $G[5]=y_1^3 y_2^2$, so all $z_i$ are zero but $z_3=12, $ hence the above  matrix has rank $4$ and it has not maximal rank accordingly with Corollary \ref{rankM}..
Since  all  the rows are not zero except the last one, it is easy to see that
  $F[4] = y_2^4$ is not in the image of $M^{[4]}(G[5])$,  as (\ref{killing}) requires.
\end{example}

\medskip
The following example  shows   that \thmref{cangradedlevel} cannot be extended to compressed type $2$ level algebras of socle degree $s=4. $

 \begin{example} \label{t=2}
 Let us consider the   forms
$G_1[4]=y_1^2 y_2 y_3$, $G_2[4]=y_1 y_2^2 y_3+ y_2y_3^3$ in $P=K[y_1, y_2, y_3] $ of degree $4$ and define in $R=K[[x_1,x_2, x_3]] $  the ideal  $$I= Ann (G_1[4]+y_3^3, G_2[4]). $$   Then $A= R/I$ is a compressed level algebra with socle degree $4, $   type $2 $ and  Hilbert function
$HF_A=\{1,3,6,6,2\}$. We prove that $A$ is not canonically graded.

We know that  $I^*= Ann (G_1[4], G_2[4])  $ and we  prove that  $A$ and $gr_{\m}(A) $ are not isomorphic as $K$-algebras.   Let $\varphi$ an analytic isomorphism sending $I $ to $I^*, $ then it is easy to see that  $\varphi = I_3$ modulo $(x_1,x_2,x_3)^2$. Following the approach of this paper, we compute the matrix $M^{[3]}(G_1[4] , G_2[4] )$ of size $20\times 18$  and,  accordingly with (\ref{killing}),  we  show  that  $y_3^3$ is not in the image of
$M^{[3]}(G_1[4] , G_2[4] )$.

\vskip 2mm
 Let $F_1[4], F_2[4]$ be two homogeneous  forms of degree $4$ of $R=K[y_1,y_2, y_3]$.
 We denote by $(z_i^j)_{i=1,\dots, 15}$ the coordinates of $F_j[4]$ with respect the basis $\Omega^*$,
 $j=1,2$.
 Then   the $20\times 18$ matrix   $M^{[3]}(F_1[4], F_2[4] )$ has  the following shape, see (\ref{multi}),

$$
{\scriptsize
\left(
\begin{array}{cccccc|cccccc|cccccc}
 3 z^1_1 & 3 z^1_2 & 3 z^1_3 & 3 z^1_4 & 3 z^1_5 & 3 z^1_6 & 0 & 0 & 0 & 0 & 0
   & 0 & 0 & 0 & 0 & 0 & 0 & 0 \\
 2 z^1_2 & 2 z^1_4 & 2 z^1_5 & 2 z^1_7 & 2 z^1_8 & 2 z^1_9 & z^1_1 & z^1_2 & z^1_3 &
   z^1_4 & z^1_5 & z^1_6 & 0 & 0 & 0 & 0 & 0 & 0 \\
 2 z^1_3 & 2 z^1_5 & 2 z^1_6 & 2 z^1_8 & 2 z^1_9 & 2 z^1_{10} & 0 & 0 & 0 & 0
   & 0 & 0 & z^1_1 & z^1_2 & z^1_3 & z^1_4 & z^1_5 & z^1_6 \\
 z^1_4 & z^1_7 & z^1_8 & z^1_{11} & z^1_{12} & z^1_{13} & 2 z^1_2 & 2 z^1_4 & 2
   z^1_5 & 2 z^1_7 & 2 z^1_8 & 2 z^1_9 & 0 & 0 & 0 & 0 & 0 & 0 \\
 z^1_5 & z^1_8 & z^1_9 & z^1_{12} & z^1_{13} & z^1_{14} & z^1_3 & z^1_5 & z^1_6 &
   z^1_8 & z^1_9 & z^1_{10} & z^1_2 & z^1_4 & z^1_5 & z^1_7 & z^1_8 & z^1_9 \\
 z^1_6 & z^1_9 & z^1_{10} & z^1_{13} & z^1_{14} & z^1_{15} & 0 & 0 & 0 & 0 & 0
   & 0 & 2 z^1_3 & 2 z^1_5 & 2 z^1_6 & 2 z^1_8 & 2 z^1_9 & 2 z^1_{10} \\ \hline
 0 & 0 & 0 & 0 & 0 & 0 & 3 z^1_4 & 3 z^1_7 & 3 z^1_8 & 3 z^1_{11} & 3
   z^1_{12} & 3 z^1_{13} & 0 & 0 & 0 & 0 & 0 & 0 \\
 0 & 0 & 0 & 0 & 0 & 0 & 2 z^1_5 & 2 z^1_8 & 2 z^1_9 & 2 z^1_{12} & 2
   z^1_{13} & 2 z^1_{14} & z^1_4 & z^1_7 & z^1_8 & z^1_{11} & z^1_{12} & z^1_{13}
   \\
 0 & 0 & 0 & 0 & 0 & 0 & z^1_6 & z^1_9 & z^1_{10} & z^1_{13} & z^1_{14} &
   z^1_{15} & 2 z^1_5 & 2 z^1_8 & 2 z^1_9 & 2 z^1_{12} & 2 z^1_{13} & 2 z^1_{14}
   \\ \hline
 0 & 0 & 0 & 0 & 0 & 0 & 0 & 0 & 0 & 0 & 0 & 0 & 3 z^1_6 & 3 z^1_9 & 3
   z^1_{10} & 3 z^1_{13} & 3 z^1_{14} & 3 z^1_{15} \\ \hline
 3 z^2_1 & 3 z^2_2 & 3 z^2_3 & 3 z^2_4 & 3 z^2_5 & 3 z^2_6 & 0 & 0 & 0 & 0 & 0
   & 0 & 0 & 0 & 0 & 0 & 0 & 0 \\
 2 z^2_2 & 2 z^2_4 & 2 z^2_5 & 2 z^2_7 & 2 z^2_8 & 2 z^2_9 & z^2_1 & z^2_2 & z^2_3 &
   z^2_4 & z^2_5 & z^2_6 & 0 & 0 & 0 & 0 & 0 & 0 \\
 2 z^2_3 & 2 z^2_5 & 2 z^2_6 & 2 z^2_8 & 2 z^2_9 & 2 z^2_{10} & 0 & 0 & 0 & 0
   & 0 & 0 & z^2_1 & z^2_2 & z^2_3 & z^2_4 & z^2_5 & z^2_6 \\
 z^2_4 & z^2_7 & z^2_8 & z^2_{11} & z^2_{12} & z^2_{13} & 2 z^2_2 & 2 z^2_4 & 2
   z^2_5 & 2 z^2_7 & 2 z^2_8 & 2 z^2_9 & 0 & 0 & 0 & 0 & 0 & 0 \\
 z^2_5 & z^2_8 & z^2_9 & z^2_{12} & z^2_{13} & z^2_{14} & z^2_3 & z^2_5 & z^2_6 &
   z^2_8 & z^2_9 & z^2_{10} & z^2_2 & z^2_4 & z^2_5 & z^2_7 & z^2_8 & z^2_9 \\
 z^2_6 & z^2_9 & z^2_{10} & z^2_{13} & z^2_{14} & z^2_{15} & 0 & 0 & 0 & 0 & 0
   & 0 & 2 z^2_3 & 2 z^2_5 & 2 z^2_6 & 2 z^2_8 & 2 z^2_9 & 2 z^2_{10} \\ \hline
 0 & 0 & 0 & 0 & 0 & 0 & 3 z^2_4 & 3 z^2_7 & 3 z^2_8 & 3 z^2_{11} & 3
   z^2_{12} & 3 z^2_{13} & 0 & 0 & 0 & 0 & 0 & 0 \\
 0 & 0 & 0 & 0 & 0 & 0 & 2 z^2_5 & 2 z^2_8 & 2 z^2_9 & 2 z^2_{12} & 2
   z^2_{13} & 2 z^2_{14} & z^2_4 & z^2_7 & z^2_8 & z^2_{11} & z^2_{12} & z^2_{13}
   \\
 0 & 0 & 0 & 0 & 0 & 0 & z^2_6 & z^2_9 & z^2_{10} & z^2_{13} & z^2_{14} &
   z^2_{15} & 2 z^2_5 & 2 z^2_8 & 2 z^2_9 & 2 z^2_{12} & 2 z^2_{13} & 2 z^2_{14}
   \\ \hline
 0 & 0 & 0 & 0 & 0 & 0 & 0 & 0 & 0 & 0 & 0 & 0 & 3 z^2_6 & 3 z^2_9 & 3
   z^2_{10} & 3 z^2_{13} & 3 z^2_{14} & 3 z^2_{15}
\end{array}
\right)
}
$$
It is enough to specialize the matrix to our case for proving that $y_3^3$ is not in the image of
$M^{[3]}(G_1[4] , G_2[4] )$.
\end{example}

\begin{remark}
Let $\mathcal C_{s,t}$ be the family of level Artin algebras of socle degree $s$ and type $t$.
This family can be parameterized by a non-empty open Zariski subset  $\mathcal I_{s,t}$ of  the affine space of $t$-uples of degree $s$ polynomials of $P$.
The previous examples suggested that if $s,t$ are not in the hypothesis of the last Theorem, then  a generic element $F \in  \mathcal I_{s,t}$ defines a {\bf non-canonically graded} level  Artin algebra
$ R/Ann_R(\langle F\rangle)$ of socle degree $s$ and type $t$.
\end{remark}



\providecommand{\bysame}{\leavevmode\hbox to3em{\hrulefill}\thinspace}
\providecommand{\MR}{\relax\ifhmode\unskip\space\fi MR }
\providecommand{\MRhref}[2]{%
  \href{http://www.ams.org/mathscinet-getitem?mr=#1}{#2}
}
\providecommand{\href}[2]{#2}

\bigskip
\noindent
Juan Elias\\
Departament d'\`Algebra i Geometria\\
Universitat de Barcelona\\
Gran Via 585, 08007 Barcelona, Spain\\
e-mail: {\tt elias@ub.edu}

\bigskip
\noindent
Maria Evelina Rossi\\
Dipartimento di Matematica\\
Universit{\`a} di Genova\\
Via Dodecaneso 35, 16146 Genova, Italy\\
e-mail: {\tt rossim@dima.unige.it}

\end{document}